\documentclass[12pt,twoside]{amsart}
\usepackage{amsmath, amsthm, amscd, amsfonts, amssymb, graphicx}
\usepackage{enumerate}
\usepackage[colorlinks=true,
linkcolor=blue,
urlcolor=cyan,
citecolor=red]{hyperref}
\usepackage{mathrsfs}
\newtheorem{theorem}{Theorem}[section]

\newtheorem{remark}{Remark}[section]

\newtheorem{proposition}{Proposition}[section]
\numberwithin{equation}{section}

\begin{document}
	
\title{On the operator Jensen inequality for convex functions}
\author{Mohsen Shah hosseini, Hamid Reza Moradi, and Baharak Moosavi}

\begin{abstract}
This paper is mainly devoted to studying operator Jensen inequality. More precisely, a new generalization of Jensen inequality and its reverse version for convex (not necessary operator convex) functions have been proved. Several special cases are discussed as well.
\end{abstract}
\maketitle
\pagestyle{myheadings}
\markboth{\centerline {M. Shah hosseini, H. R. Moradi \& B. Moosavi}}
{\centerline {On the operator Jensen inequality for convex functions}}
\bigskip
\bigskip
\section{Introduction}
Let $\mathcal{B}\left( \mathcal{H} \right)$ be the $C^*$--algebra of all bounded linear operators on a Hilbert space $\mathcal{H}$.  As customary, we reserve $m$, $M$ for scalars and ${{\mathbf{1}}_{\mathcal{H}}}$ for the identity operator on $\mathcal{H}$. A self-adjoint operator $A$ is said to be positive (written $A\ge0$) if $\left\langle Ax,x \right\rangle \ge 0$ holds for all $x\in \mathcal{H}$  also an operator $A$ is said to be strictly positive (denoted by $A>0$) if $A$ is positive and invertible. If $A$ and $B$ are self-adjoint, we
write $B\ge A$ in case $B-A\ge0$. The Gelfand map $f\left( t \right)\mapsto f\left( A \right)$ is an isometrical $*$--isomorphism between the ${{C}^{*}}$--algebra $C\left( sp\left( A \right) \right)$ of continuous functions on the spectrum $sp\left( A \right)$ of a selfadjoint operator $A$ and the ${{C}^{*}}$--algebra generated by $A$ and the identity operator ${{\mathbf{1}}_{\mathcal{H}}}$. If $f,g\in C\left( sp\left( A \right) \right)$, then $f\left( t \right)\ge g\left( t \right)$ ($t\in sp\left( A \right)$) implies that $f\left( A \right)\ge g\left( A \right)$.

For $A,B\in \mathcal{B}\left( \mathcal{H} \right)$, $A\oplus B$ is the operator defined on $\mathcal{B}\left( \mathcal{H}\oplus \mathcal{H} \right)$ by $\left( \begin{matrix}
A & 0  \\
0 & B  \\
\end{matrix} \right)$. 
 A linear map $\Phi:\mathcal{B}\left( \mathcal{H} \right)\to \mathcal{B}\left( \mathcal{K} \right)$ is positive if $\Phi \left( A \right)\ge 0$ whenever $A\ge 0$. It's said to be unital if $\Phi \left( {{\mathbf{1}}_{\mathcal{H}}} \right)={{\mathbf{1}}_{\mathcal{K}}}$. A continuous function $f$ defined on the interval $J$
is called an operator convex function if $f\left( \left( 1-v \right)A+vB \right)\le \left( 1-v \right)f\left( A \right)+vf\left( B \right)$ for every $0<v<1$ and for every pair of bounded self-adjoint operators $A$ and $B$ whose spectra are both in $J$.

The well known operator Jensen inequality states (sometimes called the Choi--Davis--Jensen inequality):
\begin{equation}\label{25}
f\left( \Phi \left( A \right) \right)\le \Phi \left( f\left( A \right) \right).
\end{equation}
It holds for every operator convex $f:J\to \mathbb{R}$, self-adjoint operator $A$ with spectra in $J$, and unital positive linear map $\Phi$ \cite{8, 9}.

Hansen et al. \cite{1} gave a general formulation of \eqref{25}. The discrete version of their result reads as follows: If $f:J\to \mathbb{R}$ is an operator convex function, ${{A}_{1}},\ldots ,{{A}_{n}}\in \mathcal{B}\left( \mathcal{H} \right)$ are self-adjoint operators with the spectra in $J$, and  ${{\Phi }_{1}},\ldots ,{{\Phi }_{n}}:\mathcal{B}\left( \mathcal{H} \right)\to \mathcal{B}\left( \mathcal{K} \right)$ are positive linear mappings such that $\sum\nolimits_{i=1}^{n}{{{\Phi }_{i}}\left( {{\mathbf{1}}_{\mathcal{H}}} \right)}={{\mathbf{1}}_{\mathcal{K}}}$, then 
\begin{equation}\label{7}
f\left( \sum\limits_{i=1}^{n}{{{\Phi }_{i}}\left( {{A}_{i}} \right)} \right)\le \sum\limits_{i=1}^{n}{{{\Phi }_{i}}\left( f\left( {{A}_{i}} \right) \right)}.
\end{equation}
Though in the case of convex function the inequality \eqref{7} does not hold in general (see \cite[Remark 2.6]{8}), we have the following estimate \cite[Lemma 2.1]{5}:
\begin{equation}\label{19}
f\left( \left\langle \sum\limits_{i=1}^{n}{{{\Phi }_{i}}\left( {{A}_{i}} \right)}x,x \right\rangle  \right)\le \left\langle \sum\limits_{i=1}^{n}{{{\Phi }_{i}}\left( f\left( {{A}_{i}} \right) \right)}x,x \right\rangle
\end{equation}
for any unit vector $x\in \mathcal{K}$. For recent results treating the Jensen operator inequality, we refer the reader to \cite{10, 6, 11}.

As a converse of \eqref{7}, in \cite{1} (see also \cite{13}), it has been shown that if $f:\left[ m,M \right]\to \mathbb{R}$ is a convex function and ${{A}_{1}},\ldots ,{{A}_{n}}$ are self-adjoint operators with the spectra in $\left[ m,M \right]$, then
\begin{equation}\label{26}
\sum\limits_{i=1}^{n}{{{\Phi }_{i}}\left( f\left( {{A}_{i}} \right) \right)}\le \beta {{\mathbf{1}}_{\mathcal{K}}}+f\left( \sum\limits_{i=1}^{n}{{{\Phi }_{i}}\left( {{A}_{i}} \right)} \right)
\end{equation}
where 
\[\beta =\max \left\{ \frac{f\left( M \right)-f\left( m \right)}{M-m}t+\frac{Mf\left( m \right)-mf\left( M \right)}{M-m}-f\left( t \right):\text{ }m\le t\le M \right\}.\]
A monograph on the reverse of Jensen inequality and its consequences is given by Furuta et al. in \cite{7}.

 In this paper, we prove an inequality of type \eqref{7} without operator convexity assumption.  Furthermore, as we can see in \eqref{26}, the constant $\beta$ is dependent on $m$ and $M$. In this paper, we establish another reverse of operator Jensen inequality by dropping this restriction.
 
\section{Operator Jensen-type inequalities without operator convexity}
Let $f:J\to \mathbb{R}$ be a convex function, $A\in \mathcal{B}\left( \mathcal{H} \right)$ self-adjoint operator with the spectra in $J$, and let $x\in \mathcal{H}$ be a unit vector. Then from \cite{12},
\begin{equation*}
f\left( \left\langle Ax,x \right\rangle  \right)\le \left\langle f\left( A \right)x,x \right\rangle.
\end{equation*}
Replace $A$ with $\Phi \left( A \right)$, where $\Phi :\mathcal{B}\left( \mathcal{H} \right)\to \mathcal{B}\left( \mathcal{K} \right)$ is a unital positive linear map, we get
\begin{equation}\label{17}
f\left( \left\langle \Phi \left( A \right)x,x \right\rangle  \right)\le \left\langle f\left( \Phi \left( A \right) \right)x,x \right\rangle
\end{equation}
for any unit vector $x\in \mathcal{K}$.
Assume that ${{A}_{1}},\ldots ,{{A}_{n}}$ are self-adjoint operators on $\mathcal{H}$ with spectra in $J$  and ${{\Phi }_{1}},\ldots ,{{\Phi }_{n}}:\mathcal{B}\left( \mathcal{H} \right)\to \mathcal{B}\left( \mathcal{K} \right)$ are positive linear maps with $\sum\nolimits_{i=1}^{n}{{{\Phi }_{i}}\left( {{\mathbf{1}}_{\mathcal{H}}} \right)}={{\mathbf{1}}_{\mathcal{K}}}$.  Now apply inequality \eqref{17} to the self-adjoint operator $A$ on the Hilbert space $\mathcal{H}\oplus \cdots \oplus \mathcal{H}$ defined by $A={{A}_{1}}\oplus \cdots \oplus {{A}_{n}}$  and the positive linear map $\Phi $ defined on $\mathcal{B}\left( \mathcal{H}\oplus \cdots \oplus \mathcal{H} \right)$  by $\Phi \left( A \right)={{\Phi }_{1}}\left( {{A}_{1}} \right)\oplus \cdots \oplus {{\Phi }_{n}}\left( {{A}_{n}} \right)$. Thus,
\begin{equation}\label{18}
f\left( \left\langle \sum\limits_{i=1}^{n}{{{\Phi }_{i}}\left( {{A}_{i}} \right)}x,x \right\rangle  \right)\le \left\langle f\left( \sum\limits_{i=1}^{n}{{{\Phi }_{i}}\left( {{A}_{i}} \right)} \right)x,x \right\rangle.
\end{equation}

Let us also recall that if $f$ is a convex function on an interval $J$, then for each point $\left( s,f\left( s \right) \right)$,
there exists a real number ${{C}_{s}}$ such that
\begin{equation}\label{1}
{{C}_{s}}\left( t-s \right)+f\left( s \right)\le f\left( t \right),\text{ }\left( t\in J \right).
\end{equation}

Inequality \eqref{18}, together with \eqref{1} yield the following theorem.
\begin{theorem}\label{4}
Let $f:J\to \mathbb{R}$ be a monotone convex function, ${{A}_{1}},\ldots ,{{A}_{n}}\in \mathcal{B}\left( \mathcal{H} \right)$ self-adjoint operators with the spectra in $J$, and let ${{\Phi }_{1}},\ldots ,{{\Phi }_{n}}:\mathcal{B}\left( \mathcal{H} \right)\to \mathcal{B}\left( \mathcal{K} \right)$ be positive linear mappings such that $\sum\nolimits_{i=1}^{n}{{{\Phi }_{i}}\left( {{\mathbf{1}}_{\mathcal{H}}} \right)}={{\mathbf{1}}_{\mathcal{K}}}$. Then 
\begin{equation}\label{5}
\sum\limits_{i=1}^{n}{{{\Phi }_{i}}\left( f\left( {{A}_{i}} \right) \right)}\le f\left( \sum\limits_{i=1}^{n}{{{\Phi }_{i}}\left( {{A}_{i}} \right)} \right)+\delta \mathbf{1}_\mathcal{K}
\end{equation}
where
\[\delta =\sup \left\{ \left\langle \sum\limits_{i=1}^{n}{{{\Phi }_{i}}\left( {{C}_{{{A}_{i}}}}{{A}_{i}} \right)}x,x \right\rangle -\left\langle \sum\limits_{i=1}^{n}{{{\Phi }_{i}}\left( {{A}_{i}} \right)}x,x \right\rangle \left\langle \sum\limits_{i=1}^{n}{{{\Phi }_{i}}\left( {{C}_{{{A}_{i}}}} \right)}x,x \right\rangle :\text{ }x\in \mathcal{K};~\left\| x \right\|=1 \right\}.\] 
\end{theorem}
\begin{proof}
Fix $t\in J$. Since $J$ contains the spectra of the ${{A}_{i}}$ for $i=1,\ldots ,n$, we may replace $s$ in the inequality \eqref{1} by ${{A}_{i}}$, via a functional calculus to get
\[f\left( {{A}_{i}} \right)\le f\left( t \right){{\mathbf{1}}_{\mathcal{H}}}+{{C}_{{{A}_{i}}}}{{A}_{i}}-t{{C}_{{{A}_{i}}}}.\]
Applying the positive linear mappings ${{\Phi }_{i}}$ and summing on $i$ from 1 to $n$, this implies
\begin{equation}\label{2}
\sum\limits_{i=1}^{n}{{{\Phi }_{i}}\left( f\left( {{A}_{i}} \right) \right)}\le f\left( t \right){{\mathbf{1}}_{\mathcal{K}}}+\sum\limits_{i=1}^{n}{{{\Phi }_{i}}\left( {{C}_{{{A}_{i}}}}{{A}_{i}} \right)}-t\sum\limits_{i=1}^{n}{{{\Phi }_{i}}\left( {{C}_{{{A}_{i}}}} \right)}.
\end{equation}
The inequality \eqref{2} easily implies, for any $x\in \mathcal{K}$ with $\left\| x \right\|=1$,
\begin{equation}\label{3}
\left\langle \sum\limits_{i=1}^{n}{{{\Phi }_{i}}\left( f\left( {{A}_{i}} \right) \right)}x,x \right\rangle \le f\left( t \right)+\left\langle \sum\limits_{i=1}^{n}{{{\Phi }_{i}}\left( {{C}_{{{A}_{i}}}}{{A}_{i}} \right)}x,x \right\rangle -t\left\langle \sum\limits_{i=1}^{n}{{{\Phi }_{i}}\left( {{C}_{{{A}_{i}}}} \right)}x,x \right\rangle.
\end{equation}
Since $\sum\nolimits_{i=1}^{n}{{{\Phi }_{i}}\left( {{\mathbf{1}}_{\mathcal{H}}} \right)}={{\mathbf{1}}_{\mathcal{K}}}$ we have $\left\langle \sum\nolimits_{i=1}^{n}{{{\Phi }_{i}}\left( {{A}_{i}} \right)}x,x \right\rangle \in J$  where $x\in \mathcal{K}$ with $\left\| x \right\|=1$. Therefore, we may replace $t$ by $\left\langle \sum\nolimits_{i=1}^{n}{{{\Phi }_{i}}\left( {{A}_{i}} \right)}x,x \right\rangle $  in \eqref{3}. This yields
\begin{equation}\label{15}
\begin{aligned}
 \left\langle \sum\limits_{i=1}^{n}{{{\Phi }_{i}}\left( f\left( {{A}_{i}} \right) \right)}x,x \right\rangle &\le f\left( \left\langle \sum\limits_{i=1}^{n}{{{\Phi }_{i}}\left( {{A}_{i}} \right)}x,x \right\rangle  \right) \\ 
&~ +\left\langle \sum\limits_{i=1}^{n}{{{\Phi }_{i}}\left( {{C}_{{{A}_{i}}}}{{A}_{i}} \right)}x,x \right\rangle -\left\langle \sum\limits_{i=1}^{n}{{{\Phi }_{i}}\left( {{A}_{i}} \right)}x,x \right\rangle \left\langle \sum\limits_{i=1}^{n}{{{\Phi }_{i}}\left( {{C}_{{{A}_{i}}}} \right)}x,x \right\rangle. 
\end{aligned}
\end{equation}
On ther other hand,
\[\begin{aligned}
 0&\le \left\langle \sum\limits_{i=1}^{n}{{{\Phi }_{i}}\left( {{C}_{{{A}_{i}}}}{{A}_{i}} \right)}x,x \right\rangle -\left\langle \sum\limits_{i=1}^{n}{{{\Phi }_{i}}\left( {{A}_{i}} \right)}x,x \right\rangle \left\langle \sum\limits_{i=1}^{n}{{{\Phi }_{i}}\left( {{C}_{{{A}_{i}}}} \right)}x,x \right\rangle  \\ 
& \le \underset{\left\| x \right\|=1}{\mathop{\underset{x\in \mathcal{K}}{\mathop{\sup }}\,}}\,\left\{ \left\langle \sum\limits_{i=1}^{n}{{{\Phi }_{i}}\left( {{C}_{{{A}_{i}}}}{{A}_{i}} \right)}x,x \right\rangle -\left\langle \sum\limits_{i=1}^{n}{{{\Phi }_{i}}\left( {{A}_{i}} \right)}x,x \right\rangle \left\langle \sum\limits_{i=1}^{n}{{{\Phi }_{i}}\left( {{C}_{{{A}_{i}}}} \right)}x,x \right\rangle  \right\}  
\end{aligned}\]
thanks to \eqref{19}. Therefore,
\begin{align}
 \left\langle \sum\limits_{i=1}^{n}{{{\Phi }_{i}}\left( f\left( {{A}_{i}} \right) \right)}x,x \right\rangle &\le f\left( \left\langle \sum\limits_{i=1}^{n}{{{\Phi }_{i}}\left( {{A}_{i}} \right)}x,x \right\rangle  \right)+\delta  \nonumber\\ 
& \le \left\langle f\left( \sum\limits_{i=1}^{n}{{{\Phi }_{i}}\left( {{A}_{i}} \right)} \right)x,x \right\rangle +\delta \quad \text{(by \eqref{18})}\nonumber.   
\end{align}
This completes the proof.
\end{proof}

\begin{remark}
Inequality \eqref{15} provides the reverse of the inequality \eqref{19}.
\end{remark}

In the next theorem, we aim to present operator Jensen-type inequality without operator convexity assumption.
\begin{theorem}\label{21}
	Let all the assumptions of Theorem \ref{4} hold, then
	\begin{equation}\label{22}
	f\left( \sum\limits_{i=1}^{n}{{{\Phi }_{i}}\left( {{A}_{i}} \right)} \right)\le \sum\limits_{i=1}^{n}{{{\Phi }_{i}}\left( f\left( {{A}_{i}} \right) \right)}+\zeta\mathbf{1}_\mathcal{K}
	\end{equation}
	where
	{\small
		\[\zeta =\sup \left\{ \left\langle {{C}_{\sum\nolimits_{i=1}^{n}{{{\Phi }_{i}}\left( {{A}_{i}} \right)}}}\sum\limits_{i=1}^{n}{{{\Phi }_{i}}\left( {{A}_{i}} \right)}x,x \right\rangle -\left\langle \sum\limits_{i=1}^{n}{{{\Phi }_{i}}\left( {{A}_{i}} \right)}x,x \right\rangle \left\langle {{C}_{\sum\nolimits_{i=1}^{n}{{{\Phi }_{i}}\left( {{A}_{i}} \right)}}}x,x \right\rangle :\text{ }x\in \mathcal{K};~\left\| x \right\|=1 \right\}.\]
	}
\end{theorem}

\begin{proof}
	Fix $t\in J$. Since $J$ contains the spectra of the ${{A}_{i}}$ for $i=1,\ldots ,n$ and $\sum\nolimits_{i=1}^{n}{{{\Phi }_{i}}\left( {{\mathbf{1}}_{\mathcal{H}}} \right)}={{\mathbf{1}}_{\mathcal{K}}}$, so the spectra of $\sum\nolimits_{i=1}^{n}{{{\Phi }_{i}}\left( {{A}_{i}} \right)}$ is also contained in $J$. Then we may replace $s$ in the inequality \eqref{1} by $\sum\nolimits_{i=1}^{n}{{{\Phi }_{i}}\left( {{A}_{i}} \right)}$, via a functional calculus to get
	\[f\left( \sum\limits_{i=1}^{n}{{{\Phi }_{i}}\left( {{A}_{i}} \right)} \right)\le f\left( t \right){{\mathbf{1}}_{\mathcal{K}}}+{{C}_{\sum\nolimits_{i=1}^{n}{{{\Phi }_{i}}\left( {{A}_{i}} \right)}}}\sum\limits_{i=1}^{n}{{{\Phi }_{i}}\left( {{A}_{i}} \right)}-t{{C}_{\sum\nolimits_{i=1}^{n}{{{\Phi }_{i}}\left( {{A}_{i}} \right)}}}.\]
	This inequality implies, for any $x\in \mathcal{K}$ with $\left\| x \right\|=1$,
	\begin{equation}\label{20}
	\left\langle f\left( \sum\limits_{i=1}^{n}{{{\Phi }_{i}}\left( {{A}_{i}} \right)} \right)x,x \right\rangle \le f\left( t \right)+\left\langle {{C}_{\sum\nolimits_{i=1}^{n}{{{\Phi }_{i}}\left( {{A}_{i}} \right)}}}\sum\limits_{i=1}^{n}{{{\Phi }_{i}}\left( {{A}_{i}} \right)}x,x \right\rangle -t\left\langle {{C}_{\sum\nolimits_{i=1}^{n}{{{\Phi }_{i}}\left( {{A}_{i}} \right)}}}x,x \right\rangle. 
	\end{equation}
	Substituting $t$ with $\left\langle \sum\nolimits_{i=1}^{n}{{{\Phi }_{i}}\left( {{A}_{i}} \right)}x,x \right\rangle $  in \eqref{20}. Thus,
	{\small
		\begin{equation}\label{23}
		\begin{aligned}
		\left\langle f\left( \sum\limits_{i=1}^{n}{{{\Phi }_{i}}\left( {{A}_{i}} \right)} \right)x,x \right\rangle &\le f\left( \left\langle \sum\limits_{i=1}^{n}{{{\Phi }_{i}}\left( {{A}_{i}} \right)}x,x \right\rangle  \right) \\ 
		&~ +\left\langle {{C}_{\sum\nolimits_{i=1}^{n}{{{\Phi }_{i}}\left( {{A}_{i}} \right)}}}\sum\limits_{i=1}^{n}{{{\Phi }_{i}}\left( {{A}_{i}} \right)}x,x \right\rangle -\left\langle \sum\limits_{i=1}^{n}{{{\Phi }_{i}}\left( {{A}_{i}} \right)}x,x \right\rangle \left\langle {{C}_{\sum\nolimits_{i=1}^{n}{{{\Phi }_{i}}\left( {{A}_{i}} \right)}}}x,x \right\rangle.
		\end{aligned}
		\end{equation}
	}
	On the other hand,
	\[\begin{aligned}
	0&\le \left\langle {{C}_{\sum\nolimits_{i=1}^{n}{{{\Phi }_{i}}\left( {{A}_{i}} \right)}}}\sum\limits_{i=1}^{n}{{{\Phi }_{i}}\left( {{A}_{i}} \right)}x,x \right\rangle -\left\langle \sum\limits_{i=1}^{n}{{{\Phi }_{i}}\left( {{A}_{i}} \right)}x,x \right\rangle \left\langle {{C}_{\sum\nolimits_{i=1}^{n}{{{\Phi }_{i}}\left( {{A}_{i}} \right)}}}x,x \right\rangle  \\ 
	& \le \underset{\left\| x \right\|=1}{\mathop{\underset{x\in \mathcal{K}}{\mathop{\sup }}\,}}\,\left\{ \left\langle {{C}_{\sum\nolimits_{i=1}^{n}{{{\Phi }_{i}}\left( {{A}_{i}} \right)}}}\sum\limits_{i=1}^{n}{{{\Phi }_{i}}\left( {{A}_{i}} \right)}x,x \right\rangle -\left\langle \sum\limits_{i=1}^{n}{{{\Phi }_{i}}\left( {{A}_{i}} \right)}x,x \right\rangle \left\langle {{C}_{\sum\nolimits_{i=1}^{n}{{{\Phi }_{i}}\left( {{A}_{i}} \right)}}}x,x \right\rangle  \right\}  
	\end{aligned}\]
	thanks to \eqref{18}. Consequently,
	\begin{align}
	\left\langle f\left( \sum\limits_{i=1}^{n}{{{\Phi }_{i}}\left( {{A}_{i}} \right)} \right)x,x \right\rangle &\le f\left( \left\langle \sum\limits_{i=1}^{n}{{{\Phi }_{i}}\left( {{A}_{i}} \right)}x,x \right\rangle  \right)+\zeta  \nonumber\\ 
	& \le \left\langle \sum\limits_{i=1}^{n}{{{\Phi }_{i}}\left( f\left( {{A}_{i}} \right) \right)}x,x \right\rangle +\zeta \quad \text{(by \eqref{19})}  \nonumber
	\end{align}
	and the proof is complete.
\end{proof}

\begin{remark}
Notice that inequality \eqref{23} can be considered as a converse of inequality \eqref{18}.
\end{remark}

\section{Some Applications}
In this section, we collect some consequences of Theorems \ref{4} and \ref{21}.

\medskip
{\bf(I)}	Suppose, in addition to the assumptions in Theorem \ref{4}, $f$ is differentiable on $J$ whose derivative $f'$ is continuous on $J$, then \eqref{5} and \eqref{22} hold with
	{\small
		\[\delta =\sup \left\{ \left\langle \sum\limits_{i=1}^{n}{{{\Phi }_{i}}\left( f'\left( {{A}_{i}} \right){{A}_{i}} \right)}x,x \right\rangle -\left\langle \sum\limits_{i=1}^{n}{{{\Phi }_{i}}\left( {{A}_{i}} \right)}x,x \right\rangle \left\langle \sum\limits_{i=1}^{n}{{{\Phi }_{i}}\left( f'\left( {{A}_{i}} \right) \right)}x,x \right\rangle :\text{ }x\in \mathcal{K};~\left\| x \right\|=1 \right\}\]
	}
	and
	\[\begin{aligned}
	\zeta &=\sup \left\{ \left\langle f'\left( \sum\limits_{i=1}^{n}{{{\Phi }_{i}}\left( {{A}_{i}} \right)} \right)\sum\limits_{i=1}^{n}{{{\Phi }_{i}}\left( {{A}_{i}} \right)}x,x \right\rangle  \right. \\ 
	&~ \left. -\left\langle \sum\limits_{i=1}^{n}{{{\Phi }_{i}}\left( {{A}_{i}} \right)}x,x \right\rangle \left\langle f'\left( \sum\limits_{i=1}^{n}{{{\Phi }_{i}}\left( {{A}_{i}} \right)} \right)x,x \right\rangle :\text{ }x\in \mathcal{K};~\left\| x \right\|=1 \right\}. 
	\end{aligned}\]

\medskip

{\bf(II)} By setting $f\left( t \right)={{t}^{p}}\left( p\ge 1\right)$ in Theorems \ref{4} and \ref{21} we find that:
	\begin{equation}\label{6}
	\sum\limits_{i=1}^{n}{{{\Phi }_{i}}\left( A_{i}^{p} \right)}\le {{\left( \sum\limits_{i=1}^{n}{{{\Phi }_{i}}\left( {{A}_{i}} \right)} \right)}^{p}}+p\delta\mathbf{1}_\mathcal{K}
	\end{equation}
	where
	\[\delta =\sup \left\{ \left\langle \sum\limits_{i=1}^{n}{{{\Phi }_{i}}\left( A_{i}^{p} \right)}x,x \right\rangle -\left\langle \sum\limits_{i=1}^{n}{{{\Phi }_{i}}\left( {{A}_{i}} \right)}x,x \right\rangle \left\langle \sum\limits_{i=1}^{n}{{{\Phi }_{i}}\left( A_{i}^{p-1} \right)}x,x \right\rangle :\text{ }x\in \mathcal{K};~\left\| x \right\|=1 \right\}\]
and
	\begin{equation}\label{24}
	{{\left( \sum\limits_{i=1}^{n}{{{\Phi }_{i}}\left( {{A}_{i}} \right)} \right)}^{p}}\le \sum\limits_{i=1}^{n}{{{\Phi }_{i}}\left( A_{i}^{p} \right)}+p\zeta\mathbf{1}_\mathcal{K}
	\end{equation}
	where
	{\small
		\[\zeta =\sup \left\{ \left\langle {{\left( \sum\limits_{i=1}^{n}{{{\Phi }_{i}}\left( {{A}_{i}} \right)} \right)}^{p}}x,x \right\rangle -\left\langle \sum\limits_{i=1}^{n}{{{\Phi }_{i}}\left( {{A}_{i}} \right)}x,x \right\rangle \left\langle {{\left( \sum\limits_{i=1}^{n}{{{\Phi }_{i}}\left( {{A}_{i}} \right)} \right)}^{p-1}}x,x \right\rangle :\text{ }x\in \mathcal{K};~\left\| x \right\|=1 \right\}\]
	}
whenever ${{A}_{1}},\ldots ,{{A}_{n}}\in \mathcal{B}\left( \mathcal{H} \right)$ are positive operators and  ${{\Phi }_{1}},\ldots ,{{\Phi }_{n}}:\mathcal{B}\left( \mathcal{H} \right)\to \mathcal{B}\left( \mathcal{K} \right)$  positive linear mappings such that $\sum\nolimits_{i=1}^{n}{{{\Phi }_{i}}\left( {{\mathbf{1}}_{\mathcal{H}}} \right)}={{\mathbf{1}}_{\mathcal{K}}}$.

If the operators ${{A}_{1}},\ldots ,{{A}_{n}}$ are strictly positive, then \eqref{6} and \eqref{24} are also true for $p<0$.

\medskip

{\bf(III)} Assume that ${{w}_{1}},\ldots ,{{w}_{n}}$ are positive scalars such that $\sum\nolimits_{i=1}^{n}{{{w}_{i}}}=1$. If we apply Theorems \ref{4} and \ref{21} for positive linear mappings ${{\Phi }_{i}}:\mathcal{B}\left( \mathcal{H} \right)\to \mathcal{B}\left( \mathcal{H} \right)$ determined by ${{\Phi }_{i}}:T\mapsto {{w}_{i}}T\left( i=1,\ldots ,n \right)$, we get
	\[\sum\limits_{i=1}^{n}{{{w}_{i}}f\left( {{A}_{i}} \right)}\le f\left( \sum\limits_{i=1}^{n}{{{w}_{i}}{{A}_{i}}} \right)+\delta {{\mathbf{1}}_{\mathcal{K}}}\]
	where
	\[\delta =\sup \left\{ \left\langle \sum\limits_{i=1}^{n}{{{w}_{i}}{{C}_{{{A}_{i}}}}{{A}_{i}}}x,x \right\rangle -\left\langle \sum\limits_{i=1}^{n}{{{w}_{i}}{{A}_{i}}}x,x \right\rangle \left\langle \sum\limits_{i=1}^{n}{{{w}_{i}}{{C}_{{{A}_{i}}}}}x,x \right\rangle :\text{ }x\in \mathcal{K};~\left\| x \right\|=1 \right\}\]
and
\[f\left( \sum\limits_{i=1}^{n}{{{w}_{i}}{{A}_{i}}} \right)\le \sum\limits_{i=1}^{n}{{{w}_{i}}f\left( {{A}_{i}} \right)}+\zeta {{\mathbf{1}}_{\mathcal{K}}}\]
where
\[\zeta =\sup \left\{ \left\langle {{C}_{\sum\nolimits_{i=1}^{n}{{{w}_{i}}{{A}_{i}}}}}\sum\limits_{i=1}^{n}{{{w}_{i}}{{A}_{i}}}x,x \right\rangle -\left\langle \sum\limits_{i=1}^{n}{{{w}_{i}}{{A}_{i}}}x,x \right\rangle \left\langle {{C}_{\sum\nolimits_{i=1}^{n}{{{w}_{i}}{{A}_{i}}}}}x,x \right\rangle :\text{ }x\in \mathcal{K};~\left\| x \right\|=1 \right\}.\]

\medskip

Choi's inequality \cite[Proposition 4.3]{3} says that
\begin{equation}\label{11}
\Phi \left( B \right)\Phi {{\left( A \right)}^{-1}}\Phi \left( B \right)\le \Phi \left( B{{A}^{-1}}B \right)
\end{equation}
whenever $B$ is self-adjoint and $A$ is positive invertible. We shall show the following complementary inequality of \eqref{11}:
\begin{proposition}\label{12}
	Let $A,B\in \mathcal{B}\left( \mathcal{H} \right)$ such that $B$ is self-adjoint and $A$ is positive invertible, and let $\Phi :\mathcal{B}\left( \mathcal{H} \right)\to \mathcal{B}\left( \mathcal{K} \right)$ be a unital positive linear mapping. Then
	\begin{equation}\label{9}
	\Phi \left( B{{A}^{-1}}B \right)\le \Phi \left( B \right)\Phi {{\left( A \right)}^{-1}}\Phi \left( B \right)+2\delta \Phi \left( A \right)
	\end{equation}
	where
	{\small
		\[\delta =\sup \left\{ \left\langle \Phi {{\left( A \right)}^{-\frac{1}{2}}}\Phi \left( B{{A}^{-1}}B \right)\Phi {{\left( A \right)}^{-\frac{1}{2}}}x,x \right\rangle -{{\left\langle \Phi {{\left( A \right)}^{-\frac{1}{2}}}\Phi \left( B \right)\Phi {{\left( A \right)}^{-\frac{1}{2}}}x,x \right\rangle }^{2}}:\text{ }x\in \mathcal{K};~\left\| x \right\|=1 \right\}.\]
	}
\end{proposition}
\begin{proof}
	It follows from Theorem \ref{4} that
	\begin{equation}\label{8}
	\Psi \left( {{T}^{2}} \right)\le \Psi {{\left( T \right)}^{2}}+2\delta {{\mathbf{1}}_{\mathcal{K}}}	
	\end{equation}
	where
	\[\delta =\sup \left\{ \left\langle \Psi \left( {{T}^{2}} \right)x,x \right\rangle -{{\left\langle \Psi \left( T \right)x,x \right\rangle }^{2}}:\text{ }x\in \mathcal{K};~\left\| x \right\|=1 \right\}.\]
	To a fixed positive $A\in \mathcal{B}\left( \mathcal{H} \right)$ we set
	\[\Psi \left( X \right)=\Phi {{\left( A \right)}^{-\frac{1}{2}}}\Phi \left( {{A}^{\frac{1}{2}}}X{{A}^{\frac{1}{2}}} \right)\Phi {{\left( A \right)}^{-\frac{1}{2}}}\]
	and notice that $\Psi :\mathcal{B}\left( \mathcal{H} \right)\to \mathcal{B}\left( \mathcal{K} \right)$ is a unital linear map. Now, if $T={{A}^{-\frac{1}{2}}}B{{A}^{-\frac{1}{2}}}$, we infer from \eqref{8} that
	\[\Phi {{\left( A \right)}^{-\frac{1}{2}}}\Phi \left( B{{A}^{-1}}B \right)\Phi {{\left( A \right)}^{-\frac{1}{2}}}\le \Phi {{\left( A \right)}^{-\frac{1}{2}}}\Phi \left( B \right)\Phi {{\left( A \right)}^{-1}}\Phi \left( B \right)\Phi {{\left( A \right)}^{-\frac{1}{2}}}+2\delta {{\mathbf{1}}_{\mathcal{K}}}\]
	where
	{\small
		\[\delta =\sup \left\{ \left\langle \Phi {{\left( A \right)}^{-\frac{1}{2}}}\Phi \left( B{{A}^{-1}}B \right)\Phi {{\left( A \right)}^{-\frac{1}{2}}}x,x \right\rangle -{{\left\langle \Phi {{\left( A \right)}^{-\frac{1}{2}}}\Phi \left( B \right)\Phi {{\left( A \right)}^{-\frac{1}{2}}}x,x \right\rangle }^{2}}:\text{ }x\in \mathcal{K};~\left\| x \right\|=1 \right\}.\]
	}
	By multiplying from the left and from the right with $\Phi {{\left( A \right)}^{\frac{1}{2}}}$ we obtain \eqref{9}.
\end{proof}
The parallel sum of two positive operators $A$, $B$ is defined as the operator
\[A:B={{\left( {{A}^{-1}}+{{B}^{-1}} \right)}^{-1}}.\]
A simple calculation shows that (see, e.g., \cite[(4.6) and (4.7)]{4})
\begin{equation}\label{13}
A:B=A-A{{\left( A+B \right)}^{-1}}A=B-B{{\left( A+B \right)}^{-1}}B.
\end{equation}

If $\Phi$ is any positive linear map, then (see \cite[Theorem 4.1.5]{4})
\begin{equation}\label{16}
\Phi \left( A:B \right)\le \Phi \left( A \right):\Phi \left( B \right).
\end{equation}
The following result gives a reverse of inequality \eqref{16}.
\begin{proposition}
	Let $A,B\in \mathcal{B}\left( \mathcal{H} \right)$ positive invertible operators and let $\Phi :\mathcal{B}\left( \mathcal{H} \right)\to \mathcal{B}\left( \mathcal{K} \right)$ be unital positive linear mapping. Then
	\[\Phi \left( A \right):\Phi \left( B \right)\le \Phi \left( A:B \right)+2\delta \Phi \left( A+B \right)\]
	where
	\[\begin{aligned}
	\delta &=\sup \left\{ \left\langle \Phi {{\left( A+B \right)}^{-\frac{1}{2}}}\Phi \left( A{{\left( A+B \right)}^{-1}}A \right)\Phi {{\left( A+B \right)}^{-\frac{1}{2}}}x,x \right\rangle  \right. \\ 
	&\quad \left. -{{\left\langle \Phi {{\left( A+B \right)}^{-\frac{1}{2}}}\Phi \left( A \right)\Phi {{\left( A+B \right)}^{-\frac{1}{2}}}x,x \right\rangle }^{2}}:\text{ }x\in \mathcal{K};~\left\| x \right\|=1 \right\}. 
	\end{aligned}\]
\end{proposition}
\begin{proof}
	Proposition \ref{12} easily implies
	\begin{equation}\label{14}
	\Phi \left( A{{\left( A+B \right)}^{-1}}A \right)\le \Phi \left( A \right)\Phi {{\left( A+B \right)}^{-1}}\Phi \left( A \right)+2\delta \Phi \left( A+B \right)
	\end{equation}
	where
	\[\begin{aligned}
	\delta &=\sup \left\{ \left\langle \Phi {{\left( A+B \right)}^{-\frac{1}{2}}}\Phi \left( A{{\left( A+B \right)}^{-1}}A \right)\Phi {{\left( A+B \right)}^{-\frac{1}{2}}}x,x \right\rangle  \right. \\ 
	&\quad \left. -{{\left\langle \Phi {{\left( A+B \right)}^{-\frac{1}{2}}}\Phi \left( A \right)\Phi {{\left( A+B \right)}^{-\frac{1}{2}}}x,x \right\rangle }^{2}}:\text{ }x\in \mathcal{K};~\left\| x \right\|=1 \right\}. 
	\end{aligned}\]
	Then we have
	\[\begin{aligned}
	\Phi \left( A \right):\Phi \left( B \right)&=\Phi \left( A \right)-\Phi \left( A \right){{\left( \Phi \left( A \right)+\Phi \left( B \right) \right)}^{-1}}\Phi \left( A \right) \quad \text{(by \eqref{13})}\\ 
	& =\Phi \left( A \right)-\Phi \left( A \right)\Phi {{\left( A+B \right)}^{-1}}\Phi \left( A \right) \quad \text{(by the linearity of $\Phi$)}\\ 
	& \le \Phi \left( A \right)-\Phi \left( A{{\left( A+B \right)}^{-1}}A \right)+2\delta \Phi \left( A+B \right) \quad \text{(by \eqref{14})}\\ 
	& =\Phi \left( A-A{{\left( A+B \right)}^{-1}}A \right)+2\delta \Phi \left( A+B \right) \quad \text{(by the linearity of $\Phi$)}\\ 
	& =\Phi \left( A:B \right)+2\delta \Phi \left( A+B \right).  
	\end{aligned}\]
	Hence the conclusions follow.
\end{proof}

\begin{remark}
A function $f:\left[ 0,\infty  \right)\to \mathbb{R}$  is called superquadratic (see \cite[Definition 1]{new0}) if for each $s\ge 0$, there exists a real constant ${{C}_{s}}$ such that	 
\begin{equation}\label{10}
f\left( \left| t-s \right| \right)+{{C}_{s}}\left( t-s \right)+f\left( s \right)\le f\left( t \right)
\end{equation}
for all $t\ge 0$.

By applying the same arguments as in Theorems \ref{4} and \ref{21} for definition \eqref{10}, one can obtain stronger estimates than \eqref{5} and \eqref{22}. 

We leave the elaboration of this idea to the interested reader.
\end{remark}

\vskip 0.3 true cm

{\tiny (M. Shah Hosseini) Department of Mathematics, Shahr-e-Qods Branch, Islamic Azad University, Tehran, Iran.}

{\tiny \textit{E-mail address:} mohsen\_shahhosseini@yahoo.com}

{\tiny \vskip 0.3 true cm }

{\tiny (H. R. Moradi) Young Researchers and Elite Club, Mashhad Branch, Islamic Azad University, Mashhad, Iran.}

{\tiny \textit{E-mail address:} hrmoradi@mshdiau.ac.ir }

{\tiny \vskip 0.3 true cm }

{\tiny (B. Moosavi) Department of Mathematics, Safadasht Branch, Islamic Azad University, Tehran, Iran.}

{\tiny \textit{E-mail address:} baharak\_moosavie@yahoo.com}


\begin{thebibliography}{9}
\bibitem{new0}
S. Abramovich, G. Jameson and G. Sinnamon, {\it Refining Jensen's inequality}, Bull. Math. Soc. Sci. Math. Roumanie., {\bf47} (2004), 3--14.
	
\bibitem{4}
R. Bhatia, {\it Positive definite matrices}, Princeton Series in Applied Mathematics, Princeton, 2007.

\bibitem{8}
M. D. Choi, {\it A Schwarz inequality for positive linear maps on $C^*$--algebras}, Illinois J. Math., {\bf18} (1974), 565--574.

\bibitem{3}
M. D. Choi, {\it Some assorted inequalities for positive linear maps on $C^*$--algebras}, J. Operator Theory., {\bf4} (1980), 271--285.

\bibitem{9}
C. Davis, {\it A Schwarz inequality for convex operator functions}, Proc. Amer. Math. Soc., {\bf8} (1957), 42--44.

\bibitem{5}
S. Furuichi, H. R. Moradi and A. Zardadi, {\it Some new Karamata type inequalities and their applications to some entropies}, Rep. Math. Phys., (2019) (accepted). arXiv:1811.07277.  

\bibitem{7}
T. Furuta, J. Mi\'ci\'c, J. Pe\v cari\'c and Y. Seo, {\it Mond--Pe\v cari\'c method in operator inequalities}, Element, Zagreb, 2005.

\bibitem{1}
F. Hansen, J. Pe\v cari\'c and I. Peri\'c, {\it Jensen's operator inequality and it's converses}, Math. Scand., {\bf100} (2007), 61--73.


\bibitem{10}
L. Horv\'ath, K. A. Khan and J. Pe\v cari\'c, {\it Cyclic refinements of the different versions of operator Jensen's inequality}, Electron. J. Linear Algebra., {\bf31}(1) (2016), 125--133.

 
\bibitem{6}
J. Mi\'ci\'c, H. R. Moradi and S. Furuichi, {\it Choi--Davis--Jensen's inequality without convexity}, J. Math. Inequal., {\bf12}(4) (2018), 1075--1085.


\bibitem{11}
 J. Mi\'ci\'c and J. Pe\v cari\'c, {\it Some mappings related to Levinson's inequality for Hilbert space operators}, Filomat., {\bf31}  (2017), 1995--2009.
 
\bibitem{13}
J. Mi\'ci\'c, J. Pe\v cari\'c and Y. Seo, {\it Complementary inequalities to inequalities of Jensen and Ando based on the Mond--Pe\v cari\'c method}, Linear Algebra Appl., {\bf318} (2000), 87--108.
 
\bibitem{12}
B. Mond and J. Pe\v cari\'c, {\it On Jensen's inequality for operator convex functions}, Houston J. Math., {\bf21} (1995), 739--753.
\end{thebibliography}
\end{document}